\documentclass[10pt, reqno, a4paper]{amsart}
\usepackage[latin1]{inputenc}
\usepackage{amsfonts}
\usepackage{amsmath}
\usepackage{amssymb}
\usepackage{amsthm}
\usepackage[T1]{fontenc}
\usepackage[dvips]{graphicx}
\usepackage{enumerate}
\usepackage{lmodern}

\usepackage{hyperref}
\usepackage[matrix, arrow]{xy}
\usepackage{listings}

\title{On base change of the fundamental group scheme}
\author{Axel St\"abler}

\DeclareMathOperator{\Proj}{Proj}

\DeclareMathOperator{\Syz}{Syz}
\DeclareMathOperator{\rk}{rk}

\DeclareMathOperator{\id}{id}

\DeclareMathOperator{\chara}{char}
\DeclareMathOperator{\Spec}{Spec}
\DeclareMathOperator{\Pic}{Pic}

\input xy
\xyoption{all}
\SelectTips{cm}{10}
\begin{document}
\swapnumbers
\theoremstyle{definition}
\newtheorem{Le}{Lemma}[section]
\newtheorem{Def}[Le]{Definition}
\newtheorem*{DefB}{Definition}
\newtheorem{Bem}[Le]{Remark}
\newtheorem{Ko}[Le]{Corollary}
\newtheorem{Theo}[Le]{Theorem}
\newtheorem*{TheoB}{Theorem}
\newtheorem{Bsp}[Le]{Example}
\newtheorem{Be}[Le]{Observation}
\newtheorem{Prop}[Le]{Proposition}
\newtheorem*{PropB}{Proposition}
\newtheorem{Sit}[Le]{Situation}
\newtheorem{Que}[Le]{Question}
\newtheorem*{Con}{Conjecture}
\newtheorem{Dis}[Le]{Discussion}
\newtheorem{Prob}[Le]{Problem}
\newtheorem{Konv}[Le]{Convention}
\def\cocoa{{\hbox{\rm C\kern-.13em o\kern-.07em C\kern-.13em o\kern-.15em
A}}}
\address{Axel St\"abler\\
Johannes Gutenberg-Universit\"at Mainz\\ Fachbereich 08\\
Staudingerweg 9\\
55099 Mainz\\
Germany}
\email{staebler@uni-mainz.de}
\date{\today}
\begin{abstract}
We provide for all prime numbers $p$ examples of smooth projective curves over a field of characteristic $p$ for which base change of the fundamental group scheme fails. This is intimately related to how $F$-trivial vector bundles, i.\,e.\ bundles trivialized by a power of the Frobenius morphism, behave in (trivial) families. We conclude with a study of the behavior of $F$-triviality in (not necessarily trivial) families.
\end{abstract}
\maketitle
\section*{Introduction}
In \cite{norifundamentalgroup} Nori introduced the so-called fundamental group scheme of a proper connected reduced scheme $X$ over a field $k$ as the affine group scheme associated to the (neutral) Tannakian category of essentially finite vector bundles on $X$ with a fixed fiber functor $x$. In \cite[Conjecture on p.\ 89]{norifundamentalgroup} Nori conjectured that if $X$ is a complete geometrically connected and reduced scheme over an algebraically closed field $k$ and $k \subseteq K$ is an arbitrary extension of algebraically closed fields then the natural map 
\begin{equation}
\label{hmorph}
h_{X}: \pi_1(X \times_k \Spec K, x) \to \pi_1(X, x) \times_k \Spec K 
\end{equation}
is an isomorphism of affine group schemes over $K$.

If $\chara k =0$ then the fundamental group scheme coincides with the \'etale fundamental group scheme so that this is well-known in this case (see \cite[Exp.\ X, Corollaire 1.8]{SGA1} or \cite[Proposition 5.6.7]{szamuelyfungroups} the latter being in English).

Recall that a vector bundle $\mathcal{S}$ on a projective variety $X$ over a field of positive characteristic is called \emph{$F$-trivial} if $F^{e^\ast} \mathcal{S} \cong \mathcal{O}_X^{\rk \mathcal{S}}$ for some $e \geq 1$, where $F: X \to X$ denotes the absolute Frobenius morphism. Note that an $F$-trivial vector bundle on a smooth projective variety is always strongly semistable with respect to any polarization (cf. \cite[Lemma 1.1]{mehtasubramanianlocalfungroupscheme}).

In order to attack Nori's conjecture Mehta and Subramanian proved (\cite{methasubramanianfungroupscheme}) the following

\begin{PropB}
Let $X$ be a reduced connected complete curve of genus $\geq 2$ over an algebraically closed field $k$ of positive characteristic, let $k \subseteq K$ be an extension of algebraically closed fields and let $x$ be a $k$-rational point. Suppose that the morphism $h_X$ in \eqref{hmorph} is an isomorphism of affine group schemes over $K$. Then for every stable, $F$-trivial bundle $\mathcal{E}$ on $X_K$, there is a bundle $\mathcal{F}$ over $X$ such that $\mathcal{F} \otimes_k K \cong \mathcal{E}$.
\end{PropB}
\begin{proof}
This is \cite[Proposition 3.1]{methasubramanianfungroupscheme}. 
\end{proof}

There they also showed that Nori's conjecture is false for curves with a cuspidal singularity (see \cite[paragraph after Proposition 3.1]{methasubramanianfungroupscheme}). This proposition will be crucial for our purposes.
Mehta and Subramanian also established several other equivalent notions in a follow-up paper (see \cite{mehtasubramanianlocalfungroupscheme} and, in particular, [Theorem on p.\, 208, ibid.]).

In \cite{paulycounterexamplenori}, Pauly showed that Nori's conjecture is false for certain smooth projective ordinary curves of genus $2$ in characteristic $2$. Pauly's method is to use the Proposition above and to explicitly determine all stable rank $2$ bundles trivialized by the fourth Frobenius pull back. More precisely, Pauly exploits that one has a very explicit description of the moduli space of semistable rank $2$ bundles over these curves and analyzes the (iterated) preimages of the Verschiebung.

In this paper, we will provide for all characteristics $p$ examples of smooth projective curves where base change fails for the fundamental group scheme.
Our method is rather different from that of Pauly. We will use so-called syzygy bundles to explicitly construct bundles with the looked-for behavior. A \emph{syzygy bundle} is a locally free sheaf $\mathcal{S}$ which fits into a short exact sequence \[
\begin{xy}
\xymatrix{0 \ar[r]& \mathcal{S} \ar[r]& \bigoplus_{i=0}^{n} \mathcal{O}_Y(-d_i) \ar[r]& \mathcal{O}_Y \ar[r] & 0,}
\end{xy} 
\] where $\mathcal{O}_Y(1)$ is an ample line bundle on a projective variety $Y$. We call a global section of $\mathcal{S}(m)$ a \emph{syzygy of total degree $m$}. The advantage of syzygy bundles is that they are well-suited for explicit computations. On the other hand, on a smooth projective curve over an algebraically closed field any vector bundle is, up to twist, a syzygy bundle for a suitable polarization (cf.\ e.\,g.\ \cite[Proposition 3.8]{brennerstaeblerdaggersolid}).

The typical example the reader should keep in mind is that $R$ is a standard-graded normal two-dimensional domain and $I$ a homogeneous $R_+$-primary ideal. Then the sheaf associated to the first syzygies of $I$ on $\Proj R$ is a syzygy bundle. In particular, if $I = (f_0, \ldots, f_n)$, where the $f_i$ are homogeneous of degrees $d_i$ then we also write $\mathcal{S} = \Syz(f_0, \ldots, f_n)$ for the bundle on $\Proj R$ obtained by sheafifying the exact sequence of (graded) $R$-modules \[\begin{xy} \xymatrix@1{0 \ar[r]& \Syz(f_0, \ldots, f_n) \ar[r]& \bigoplus_{i=0}^n R(-d_i)  \ar[r]^<<<<<{\varphi} & R \ar[r] &R/(f_0, \ldots, f_n) \ar[r] & 0} \end{xy},\] where $\varphi = (f_0, \ldots, f_n)$. Note that the penultimate term vanishes on $\Proj R$ since $I$ is $R_+$-primary.

Throughout this paper we will often consider morphisms \[X = Y \times_k \Spec k[t] \longrightarrow \Spec k[t]\] and a vector bundle $\mathcal{S}$ over $X$. We then use notation $\mathcal{S}_t, X_t$ to denote the generic fibers and $\mathcal{S}_{\sigma}, X_\sigma$ for $\sigma \in k$ will denote the special fiber over the point $\sigma$.

Combining results of Biswas-dos Santos (\cite[Proposition 11 and Theorem 13]{biswasdossantosfinitemorphismfungroupscheme2}) and Mehta-Subramanian (\cite[Theorem on p.\ 208]{mehtasubramanianlocalfungroupscheme}) these examples also yield new examples where the universal torsor for the fundamental group scheme (cf.\ \cite[Chapter II, Definition 1 and 3]{norifundamentalgroup}) is not reduced.

\section{Stable rank two bundles}

In this section we collect several results on $F$-trivial and on semistable vector bundles that we shall need in the sequel.

\begin{Le}
Let $X$ be a smooth projective curve over a field $k$ of positive characteristic $p$. Let $\mathcal{S}$ be a locally free sheaf of rank $2$ with trivial determinant. Assume that $F^{e^\ast} \mathcal{S}$ is trivial for some $e \geq 1$. Then $\mathcal{S}$ is stable if and only if $\mathcal{S}$ is not an extension of a $p^e$-torsion line bundle with its dual.
\end{Le}
\begin{proof}
We have that $\mathcal{S}$ is (strongly) semistable since it is trivialized by a suitable Frobenius pull back. Assume that $\mathcal{S}$ is not stable. That is we have a line bundle $\mathcal{L}$ of degree zero and an injective morphism $\mathcal{L} \to \mathcal{S}$. Since $\mathcal{S}$ is semistable the cokernel $\mathcal{G}$ is locally free. As the determinant of $\mathcal{S}$ is trivial we must have $\mathcal{G} \cong \mathcal{L}^\vee$. Consider now the cohomology of the pull backed short exact sequence:
\[
 \begin{xy}
  \xymatrix{0 \ar[r] & H^0(X, F^{e^\ast} \mathcal{L}) \ar[r] & H^0(X,F^{e^\ast}\mathcal{S}) \ar[r] & H^0(X, F^{e^\ast}\mathcal{L}^\vee) \ar[r] & \ldots}
 \end{xy}
\]
If $\mathcal{L}$ is not $p^e$-torsion $F^{e^\ast} \mathcal{L}$ has no global sections. Hence, we must have that $\dim_k H^0(X, F^{e^\ast}\mathcal{L}^\vee) \geq 2$. This is a contradiction since $F^{e^\ast}\mathcal{L}^\vee$ is a line bundle of degree zero.
\end{proof}

\begin{Le}
\label{LTrivial}
Let $X$ be a smooth projective curve over a field $k$. Let $\mathcal{S}$ be a locally free sheaf of rank $n$ with trivial determinant. If $\dim_k H^0(X,\mathcal{S}) \geq n$ and $\mathcal{S}$ is semistable then $\mathcal{S} \cong \mathcal{O}_X^n$. The same conclusion holds if $\dim_k H^0(X, \mathcal{S}) \geq n$ and if there are $n-1$ linearly independent sections $s_1, \ldots, s_{n-1}$ such that $s_1 \wedge \ldots \wedge s_{n-1} \in \bigwedge^{n-1} \mathcal{S}$ has no zeros.
\end{Le}
\begin{proof}
Fix $n-1$ linearly independent global sections $s_1, \ldots, s_{n-1}$. These yield an injective morphism $\mathcal{O}_X^{n-1} \to \mathcal{S}$. By semistability of $\mathcal{S}$ the cokernel is locally free (if the section $s_1 \wedge \ldots \wedge s_n$ has no zeros this also holds). As $\det \mathcal{S} = \mathcal{O}_X$ the cokernel is furthermore isomorphic to $\mathcal{O}_X$. We therefore have a short exact sequence
\[
\begin{xy}
 \xymatrix{0 \ar[r] & \mathcal{O}_X^{n-1} \ar[r] & \mathcal{S} \ar[r]^p &\mathcal{O}_X \ar[r] & 0.}
\end{xy}
 \]
Fix a global section $t$ which is not in the $k$-span of the $s_1, \ldots, s_{n-1}$. Then $p \circ t$ is an isomorphism since it is not the zero map. Hence, $p$ has a section and the sequence splits.
\end{proof}

The following lemma is well-known but we include a proof for the convenience of the reader.

\begin{Le}
\label{SufficientStable}
Let $Y$ be a smooth projective curve over a field $k$ of positive characteristic $p > 0$. Let $\mathcal{S}$ be a rank two vector bundle on $Y$ that is semistable but not strongly semistable. Then $\mathcal{S}$ is stable.
\end{Le}
\begin{proof}
If $\mathcal{S}$ is semistable and $\deg \mathcal{S}$ is odd then it is stable since then the rank and degree are coprime. If $\deg \mathcal{S}$ is even we may assume, after twisting, that $\deg \mathcal{S} = 0$.

We will show that if $\mathcal{S}$ is semistable but not stable then it is strongly semistable. After twisting with a degree zero line bundle we may assume that we have a short exact sequence $0 \to \mathcal{O}_Y \to \mathcal{S} \to \mathcal{L} \to 0$, where $\mathcal{L}$ is a line bundle. Consider the $e$th Frobenius pull back of this sequence and let $\mathcal{G} \subset F^{e^\ast}\mathcal{S}$ be a line bundle. The morphism $\mathcal{G} \to F^{e^\ast}\mathcal{L}$ is either zero or injective. In the first case we must have that $\mathcal{G}$ injects into $\mathcal{O}_Y$, hence has degree $\leq 0$. In the second case, it injects into $F^{e^\ast}\mathcal{L}$, hence its degree must likewise be $\leq 0$.
\end{proof}

The key tool to ensure that a strongly semistable bundle defined over $X_K$ does not descend to a bundle over $X_k$, where $k \subseteq K$ is an extension of algebraically closed fields and $X$ a smooth projective curve, is the following

\begin{Prop}
\label{NonConstantFamily}
Let $X$ be a smooth projective curve over an algebraically closed field $k$ of positive characteristic. Let $V$ be a discrete valuation ring containing $k$ with maximal ideal $(t)$, residue field $k$ and quotient field $K$. Let $\mathcal{S}$ be a coherent torsion free sheaf on $X \times_k V$ such that
\begin{enumerate}[(i)]
 \item{The restriction $\mathcal{S}_t$ of $\mathcal{S}$ to the generic fiber $X_K$ is locally free and strongly semistable.}
\item{The restriction $\mathcal{S}_0$ of $\mathcal{S}$ to the special fiber $X_k$ is locally free and stable but not strongly semistable.} 
\end{enumerate}
Then $\mathcal{S}_t$ does not stem from the special fiber, i.\,e.\ there is no locally free sheaf $\mathcal{F}$ on $X \times_k V/(t)$ such that $p^\ast \mathcal{F} \cong \mathcal{S}_t$, where $p: X \times_k K \to X$ is the base change morphism.
\end{Prop}
\begin{proof}
Denote by $j: X \times_k V/(t) \to X \times_k V$ the closed immersion of the special fiber and by $i: X \times_k K \to X \times_k V$ the open immersion of the generic fiber. Also note that we have a morphism $c: X \times_k V \to X \times_k V/(t)$ due to the $k$-algebra structure on $V$. In particular, $c \circ j = \id_{X \times_k V/(t)}$ and $c \circ i = p$.

Assume now that such $\mathcal{F}$ exists.
By construction $c^\ast \mathcal{F}$ and $\mathcal{S}$ agree on generic fibers. Since $\mathcal{S}_0$ is stable we must have by \cite[Theorem on page 99]{langton} (or \cite[Proposition 28.6]{hartshornedeformation}) that $\mathcal{S}_0 \cong \mathcal{F}$. But $\mathcal{F}$ has to be strongly semistable since $p^\ast \mathcal{F} = \mathcal{S}_t$ is. Since $\mathcal{S}_0$ is not strongly semistable we obtain the desired contradiction.
\end{proof}

\begin{Bem}
We would like to mention that Mehta and Subramanian have an unpublished preprint where they explore the relation of the failure of Nori's conjecture with the non-properness of the functor of $F$-trivial bundles.
\end{Bem}

\section{An example in characteristic $2$}

In this section we provide an infinite family of curves in $\mathbb{P}^2_k$, $k$ an algebraically closed field of characteristic $2$, where base change of the fundamental group scheme fails.

\begin{Theo}
\label{TFTrivial}
Consider the syzygy bundle $\mathcal{S} = \Syz(x^2, y^2, tz^2 + s xy)(3)$ on $X_{n,l} \times_{k} \Spec k[s,t]$, where $k$ is a field of characteristic $2$,
$n \geq 2$ and $0 \leq l \leq \lfloor\frac{2^{n-1}}{3} - \frac{1}{2}\rfloor$ are integers and $X_{n,l}$ is the smooth projective curve given by
\[ \Proj k[x,y,z]/(x^{2^n + 2l + 1} + y^{2^n + 2l + 1} + z^{2^n + 2l +1}).\] 
Then $\mathcal{S}$ restricted to the generic fiber is stable and trivialized by the $n$th Frobenius pull back.
The trivializing syzygies are
\begin{align*}
s_1 &= (x^{3 + 6l} z^{2^n - 3 - 6l} t^{2^{n+1}},\, x^{1 + 2l} y^{2 + 4l} z^{2^n - 3 - 6l} t^{2^n + 1} + y^{1 + 2l} z^{2^n - 2l - 1} t^{2^n} s\\& + x^{2^n} s^{2^n + 1},\, x^{1+2l} z^{2^n - 2l - 1} t^{2^n} + s^{2n} y^{2^n}),\\
s_2 &=(y^{1 + 2l} x^{2 + 4l} z^{2^n - 3 - 6l} t^{2^n + 1} + x^{1 + 2l} z^{2^n - 2l - 1} t^{2^n} s+ y^{2^n} s^{2^n +1},\\& y^{3 + 6l} z^{2^n - 3 - 6l} t^{2^{n+1}},\, y^{1+2l} z^{2^n - 2l - 1} t^{2^n} + s^{2^n} x^{2^n}).  
\end{align*}
\end{Theo}
\begin{proof}
We fix a Fermat curve $X$ as in the theorem and omit the index.
First we prove that $\mathcal{S}$ is strongly semistable on the generic fiber. On the special fiber $s= 1, t = 0$ we obtain $\mathcal{S}_{1,0} = \Syz(x^2, y^2, xy)(3)$ which is already defined on $\mathbb{P}^1_k = \Proj k[x,y]$ and splits as $\mathcal{O}^2_X$ (the bundle admits the global sections $(-y,0,x)$ and $(0,-x,y)$, now use Lemma \ref{LTrivial}). Hence, $\mathcal{S}_{1,0}$ is strongly semistable. By \cite[Corollary 3.12 and the remark before Proposition 5.2]{miyaokachern} this then also holds generically.

One easily verifies that $s_1$ and $s_2$ are indeed syzygies of $F^{n^\ast} \mathcal{S}$. As they are generically $k$-linearly independent (set $t =0, s =1$ and look at the last component of $s_1, s_2$) Lemma \ref{LTrivial} implies that $F^{n^\ast} \mathcal{S}$ is trivial on the generic fiber. Moreover, these syzygies are well-defined (i.\,e.\,the exponents are natural numbers) precisely if $0 \leq l \leq \lfloor\frac{2^{n-1}}{3} - \frac{1}{2}\rfloor$. 

Finally we show that $\mathcal{S}$ is generically stable. Consider the fiber $s=0, t=1$. Note that $2^n < 2^n + 2l + 1 \leq \frac{4}{3} 2^n < 3 \cdot 2^{n-1}$, so applying \cite[Corollary 2]{brennermiyaoka} we see that $\mathcal{S}_{0,1}$ is not strongly semistable. It is however semistable by \cite[Proposition 6.2]{brennercomputationtight}. Hence, Lemma \ref{SufficientStable} yields that it is (geometrically) stable. Being geometrically stable is an open property (e.\,g.\ by \cite[Proposition 2.3.1]{huybrechtslehn}) hence this also holds generically.
\end{proof}

\begin{Ko}
\label{KoFtrivial}
Let $n \geq 2$ be an integer and $0 \leq l \leq \lfloor\frac{2^{n-1}}{3} - \frac{1}{2}\rfloor$ and consider the smooth projective curve \[X = \Proj k[x,y,z]/(x^{2^n + 2l + 1} + y^{2^n + 2l + 1} + z^{2^n  + 2l + 1}),\] where $k$ is a field of characteristic $2$. Consider the syzygy bundle \[\mathcal{S} = \Syz(x^2, y^2, tz^2 + (1-t)xy)(3) \text{ over } X \times_k \Spec k[t]_{(t-1)}.\] Then $\mathcal{S}$ is generically stable and trivialized by the $n$th Frobenius pull back while $\mathcal{S}$ restricted to the special fiber is stable but not strongly semistable.   
\end{Ko}
\begin{proof}
The bundle $\mathcal{S}$ is obtained from the syzygy bundle in Theorem \ref{TFTrivial} by specializing $s \mapsto 1 -t$ and localizing $X \times \Spec k[t]$ at $(t-1)$.

Generically the global sections $s_1, s_2$ are then still linearly independent (again set $t=0$ and look at the last component) and $\mathcal{S}_t$ is strongly semistable so that $F^{n^\ast} \mathcal{S}_t$ is trivial. Likewise, the same argument as in Theorem \ref{TFTrivial} shows that $\mathcal{S}$ is generically stable.

On the special fiber we obtain $\mathcal{S}_{0} = \Syz(x^2, y^2, z^2)(3)$. Note that this is the bundle $\mathcal{S}_{0,1}$ considered in the last part of the proof of Theorem \ref{TFTrivial} which was shown to be stable but not strongly semistable.
\end{proof}

\begin{Bem}
We would like to point out that there are further cases, where the bundle $\mathcal{S} = \Syz(x^2, y^2, tz^2 + (1-t)xy)(3)$ on certain Fermat curves is generically stable and $F$-trivial and stable but not strongly semistable on the special fiber.

For instance, consider the Fermat curve $Y$ given by the equation $x^{11} + y^{11} = z^{11}$. Then by similar arguments as in Theorem \ref{TFTrivial} one can show that $\mathcal{S}$ on $Y \times_k \Spec k[t]_{(t-1)}$ is generically strongly semistable and stable but not strongly semistable on the special fiber. Moreover, one verifies that $F^{5^\ast}\mathcal{S}_{t}$ has the $\overline{\mathbb{F}}_2$-linearly independent global sections

\begin{align*}
s_1 &= (y^{32}t^{128} + x^{26}y^3z^3t^{128} + x^{14}y^2z^{16}t^{128} + x^2y^{12}z^{18}t^{128} + x^{12}z^{20}t^{128} \\&+ x^{14}y^2z^{16}t^{96} + x^{12}z^{20}t^{96} + x^2y^{12}z^{18}t^{64} + x^{12}z^{20}t^{64} + x^{12}z^{20}t^{32} + y^{32},\\& y^{25}z^7t^{128} + xy^{15}z^{16}t^{128} + x^{11}y^3z^{18}t^{128} + x^2y^{22}z^3t^{128} + xy^{15}z^{16}t^{96} + y^{25}z^7t^{64} \\&+ x^{11}y^3z^{18}t^{64},\, x^{32}t^{96} + x^{24}y^3z^5t^{96} + x^2y^{25}z^5t^{96} + x^{23}y^2z^7t^{96} + x^{12}y^{13}z^7t^{96} \\&+ xy^{24}z^7t^{96} + y^{12}z^{20}t^{96} + x^{32}t^{64} + x^{23}y^2z^7t^{64} + x^{12}y^{13}z^7t^{64}+ xy^{24}z^7t^{64}\\& + x^{32}t^{32} + y^{12}z^{20}t^{32} + x^{32}),
\end{align*}
\begin{align*}
s_2 &= (x^{27}y^2z^3t^{128} + x^{14}y^{11}z^7t^{128} + x^{15}yz^{16}t^{128} + x^3z^{29}t^{128} + x^{15}yz^{16}t^{96} \\&+ x^{14}y^{11}z^7t^{64} + x^3z^{29}t^{64},\, x^{32}t^{128} + x^3y^{26}z^3t^{128} + x^{23}y^2z^7t^{128} + x^{12}y^{13}z^7t^{128} \\&+ x^2y^{14}z^{16}t^{128} + y^{12}z^{20}t^{128} + x^2y^{14}z^{16}t^{96} + y^{12}z^{20}t^{96} + x^{23}y^2z^7t^{64}\\ &+ x^{12}y^{13}z^7t^{64} + y^{12}z^{20}t^{64} + y^{12}z^{20}t^{32} + x^{32},\, y^{32}t^{96} + x^{25}y^2z^5t^{96}\\& + x^3y^{24}z^5t^{96} + x^{24}yz^7t^{96} + x^{13}y^{12}z^7t^{96} + x^2y^{23}z^7t^{96} + x^{12}z^{20}t^{96} + y^{32}t^{64}\\& + x^{24}yz^7t^{64} + x^{13}y^{12}z^7t^{64} + x^2y^{23}z^7t^{64} + y^{32}t^{32} + x^{12}z^{20}t^{32} + y^{32})
\end{align*} of total degree $96$.

Likewise, one can show that $F^{9^\ast}\mathcal{S}$ is generically trivial on the relative Fermat curves over $\mathbb{F}_2[t]$ of degrees $171, 173, 175$ and $177$. And again, as in Corollary \ref{KoFtrivial} one sees that $\mathcal{S}$ is generically strongly semistable and stable but not strongly semistable on the special fiber.
\end{Bem}

\begin{Theo}
\label{p=2Theorem}
Consider the smooth projective curve \[X = \Proj k[x,y,z]/(x^{2^n + 2l + 1 } + y^{2^n + 2l + 1} + z^{2^n + 2l + 1})\] with $n \geq 2$ arbitrary and $0 \leq l \leq \lfloor\frac{2^{n-1}}{3} - \frac{1}{2}\rfloor$, where $k$ is an algebraically closed field with $\chara k = 2$. Denote by $K$ an algebraic closure of $k(t)$. Then the natural morphism \[h_{X}: \pi_1(X \times_k \Spec K, x) \to \pi_1(X, x) \times_k \Spec K\]  of affine group schemes over $K$ is not an isomorphism.
\end{Theo}
\begin{proof}
This follows from combining Theorem \ref{TFTrivial}, Proposition \ref{NonConstantFamily} and \cite[Proposition 3.1]{methasubramanianfungroupscheme} (cf.\ Introduction).
\end{proof}

\begin{Bem}
For $X$ a Fermat curve as in Theorem \ref{p=2Theorem} above the \v{C}ech cohomology class \[\frac{z^{2^{n-1}  +l + 1}}{x^{2^{n-2}}y^{2^{n-2} +l+ 1}} \in H^1(X, \mathcal{O}_X)\] is nonzero but its first Frobenius pull back is \[\frac{z^{2^n + 2l+2 }}{x^{2^{n-1}}y^{2^{n-1} + 2l+2 }} = \frac{z x^{2^{n -1} + 2l + 1}}{y^{2^{n-1} + 2l+ 2}} + \frac{z y^{2^{n -1} -1}}{x^{2^{n-1}}} = 0.\]
Hence, the curves in Theorem \ref{p=2Theorem} are all non-ordinary.
\end{Bem}

\section{Examples for all prime characteristics}
In this section we provide examples of failure of base change of the fundamental group scheme for all prime characteristics.

\begin{Le}
\label{PSyzfuerallepgeq3}
Let $p \geq 2$ be a prime number and $k$ a field of characteristic $p$ and $l \in \mathbb{N}$ with $1 \leq l \leq \lfloor\frac{p}{2}\rfloor$. Consider the smooth projective relative curve \[Y = \Proj k[t][x,y,z]/(x^{3pl-1} + y^{3pl-1} + z^{3pl-1} + x^{pl} z^{2pl-1}) \longrightarrow \Spec k[t]\] and the twisted syzygy bundle $\mathcal{S}(3l) = \Syz(x^{2l}, y^{2l}, z^{2l} + t x^l z^l)(3l)$ on $Y$. Then $\mathcal{S}(3l)$ restricted to the generic fiber is stable and $F^\ast(\mathcal{S}(3l))$ is, over the generic fiber, a non-split extension of $\mathcal{O}_{Y_t}$ with itself. The extension is given (up to multiplication by an element of $k(t)^\times$) by the cohomology class \[c = \sum_{i=1}^{pl} \frac{y^{2pl}}{z^i x^{2pl -i}} (t^p-1)^{i-1}(t^p-t^{2p})^{i-1} + \sum_{i=1}^{pl-1} \frac{y^{2pl}}{z^{pl+i}x^{pl-i}} (t^p-1)^i (t^p-t^{2p})^{i-1}\] in $H^1(Y_t, \mathcal{O}_{Y_t})$ using \v{C}ech cohomology for the covering $D_+(x), D_+(z)$. Moreover, $\mathcal{S}(3l)_1$ is stable but not strongly semistable.
\end{Le}
\begin{proof}
The syzygy bundle $F^\ast \mathcal{S}_1 = \Syz(x^{2pl}, y^{2pl}, z^{2pl} + x^{pl}z^{pl})$ admits the syzygy $s = (x^{pl-1}, y^{pl-1}, z^{pl-1})$ of total degree $3pl-1$. Hence, $s$ yields a destabilizing subbundle since the slope of $F^\ast \mathcal{S}_1$ is $-3pl \deg \mathcal{O}_{Y_1}(1)$. By \cite[Lemma 2.3]{brennerstaeblergeometricdeformations} and Lemma \ref{SufficientStable} we obtain that $\mathcal{S}(3l)_1$ is (geometrically) stable but not strongly semistable.
It follows by openness of geometric stability that $\mathcal{S}(3l)_t$ is geometrically stable.

Note that $F^\ast(\mathcal{S}(3l)_t) = \Syz(x^{2pl}, y^{2pl}, z^{2pl} + t^px^{pl}z^{pl})(3)_t$ has the section \[s_1 = ((t^p-t^{2p})z^{pl} - x^{pl-1}z, -y^{pl-1}z, (t^p-1)x^{pl} - z^{pl})\] without zeros. Indeed, if $z =0$ we immediately obtain that $x= y =0$. Conversely, if $z \neq 0$ then specialising $t=0$ we see that there are no points on the curve where the section vanishes.

The section $s_1$ thus yields a short exact sequence $0 \to \mathcal{O}_{Y_t} \xrightarrow{s_1} F^\ast (\mathcal{S}(3l)_t) \to \mathcal{O}_{Y_t} \to 0$. This extension is given by the cohomology class $c = \delta(1) \in H^1(Y_t, \mathcal{O}_{Y_t})$, where $\delta$ is the connecting homomorphism.

Since we only need to know $c$ up to multiplication by a unit in $k(t)$ it is enough to determine the kernel of $s_1: H^1(Y_t,\mathcal{O}_{Y_t}) \to H^1(Y_t,F^\ast(\mathcal{S}(3)_t))$. To fix notation we work with \v{C}ech cohomology using the cover $D_+(x), D_+(z)$ so that a basis of $H^1(Y_t,\mathcal{O}_{Y_t})$ is given by the $\frac{y^{a+b}}{x^az^b}$ where $a+b \leq 3pl -2$ and $a, b \geq 1$. Clearly, these elements are linearly independent. By the genus formula for smooth plane curves $\dim H^1(Y_t,\mathcal{O}_{Y_t}) = \frac{(3pl-2)(3pl -3)}{2}$ which is the amount of occuring elements\footnote{Let $b$ be in the range $1$ to $3pl-3$ then we have $3pl-3, 3pl - 2, \ldots, 1$ possibilities for $a$. Taking the sum over these yields $\sum_{i=1}^{3pl-3} i = \frac{(3pl -3)(3pl - 2)}{2}$}.

Claim 1: $c$ is mapped to zero along the composition \[H^1(Y_t, \mathcal{O}_{Y_t}) \longrightarrow H^1(Y_t, F^\ast(\mathcal{S}(3l)_t)) \longrightarrow H^1(Y_t, \mathcal{O}_{Y_t}(pl)^3)\] (the latter map is obtained from taking cohomology of the presenting sequence of $\mathcal{S}(3l)_t$). We have to show that $s_1 \cdot c$ is a coboundary, i.\,e.\ that there are tuples $g = (g_1, g_2, g_3), h = (h_1, h_2, h_3)$ such that $s_1 c  = \frac{g}{x^a} - \frac{h}{z^b}$ for suitable $a,b$. This is an easy calculation which will be left to the reader.

Claim 2: The cohomology class $s_1 c$ is already a coboundary in $\mathcal{S}$.  Denote $(x^{2pl}, y^{2pl}, z^{2pl} + t^p x^{pl} z^{pl})$ by $f$. For the claim to hold we have to verify that there are $\frac{g'}{x^{a'}}, \frac{h'}{z^{b'}}$ as above, such that in addition $\frac{g'}{x^{a'}} \cdot f =0$ and $\frac{h'}{z^{b'}} \cdot f = 0$. Since $s_1 c$ is a syzygy it suffices to show that $\frac{g}{x^a} \cdot f = P(x,y,z) y^{2pl}$, where $P$ is a polynomial in $x,y,z$. For then $\frac{g'}{x^a} := \frac{g}{x^a} - (0,P(x,y,z), 0)$ and $\frac{h'}{z^b} := \frac{h}{z^b} + (0, P(x,y,z), 0)$ will do. Note that $\frac{g}{x^a} \cdot f$ is a polynomial of degree $3pl$ in  $x,y,z$ (use \cite[Corollaire III.3.5]{SGA2}).
For the first and third component of $c s_1$ the variable $y$ occurs exactly with degree $2pl$. In the second component of $c s_1$ the variable occurs with degree $y^{3pl-1}$ so that after multiplication with $f_2$ and normalizing using the curve equation it also occurs exactly with degree $2pl$.

Finally, note that the vectors occuring in the linear combination that defines $c$ are all part of a basis of $H^1(Y_t, \mathcal{O}_{Y_t})$. So this extension does indeed not split.
\end{proof}

\begin{Bem}
Note that a similar example occurs with $l = 1$ in \cite[Remark 5.1]{kaidkasprowitz} in a different context.

One can also verify directly that for $p = 2$ (and $l =1$) the action of Frobenius on $H^1(Y_t, \mathcal{O}_{Y_t})$ is nilpotent. Hence, $\mathcal{S}(3)_t$ is $F$-trivial for $p = 2$.
\end{Bem}

Recall that Frobenius acts on $H^1(Y_t, \mathcal{O}_{Y_t})$ and that for a vector space $V$ over a perfect field $k$ with Frobenius action one has the so-called Fitting decomposition $V = V_s \oplus V_n$, where $F$ acts nilpotently on $V_n$ and bijectively on $V_s$ (cf.\ \cite[XXII]{SGA7II} or \cite[III.6]{hartshorneamplesubvarieties}). In particular, there exists a basis of $V_s$ which is (pointwise) fixed by $F$ (cf. \cite[Proposition XXII.1.1]{SGA7II}).

In what follows, we will pass to an \'etale cover of the curve defined in Lemma \ref{PSyzfuerallepgeq3} such that the bundle generically becomes $F$-trivial (i.\,e.\ we will annihilate the semisimple part of the extension class $c$ by an Artin-Schreier extension). In order to ensure that the bundle is still generically stable after this \'etale pull back we need that the nilpotent part of $c$ is non-trivial. This is the claim of the next lemma.

\begin{Le}
\label{Lnichtreinetaletrivialisierbar}
Let $Y' = \Proj K[x,y,z]/(x^{3pl-1} + y^{3pl-1} + z^{3pl-1} + x^{pl} z^{2pl-1})$, where $K$ is a perfect closure of $k(t)$, $k$ a field of characteristic $p \geq 3$ and  $1 \leq l < \frac{p}{2}$ an integer. Denote by $c$ the \v{C}ech-cohomology class \[\sum_{i=1}^{pl} \frac{y^{2pl}}{z^i x^{2pl -i}} (t^p-1)^{i-1}(t^p-t^{2p})^{i-1} + \sum_{i=1}^{pl-1} \frac{y^{2pl}}{z^{pl+i}x^{pl-i}} (t^p-1)^i (t^p-t^{2p})^{i-1}\] in $H^1(Y', \mathcal{O}_{Y'})$
using the open cover $D_+(x), D_+(z)$.
Then the projection of $c$ to $H^1(Y', \mathcal{O}_{Y'})_n$ is nontrivial. In particular, $Y'$ is not ordinary.
\end{Le}
\begin{proof}
Note that $Y'$ is obtained from the generic fiber of the relative curve $Y$ in Lemma \ref{PSyzfuerallepgeq3} by base change. Since the family in Lemma \ref{PSyzfuerallepgeq3} is trivial we have $H^1(Y', \mathcal{O}_{Y'}) = H^1(Y_0, \mathcal{O}_{Y_0}) \otimes K$. Passing to a perfect closure of $k$ and base changing we may assume that $k$ is perfect\footnote{Actually, we could also work over $\mathbb{F}_p$ and base change to $k$ and $K$.}. We therefore find a Fitting decomposition of $H^1(Y_0, \mathcal{O}_{Y_0})$. By base change to $K$ we obtain a Fitting decomposition of $H^1(Y',\mathcal{O}_{Y'})$. The point is that we therefore find bases of $H^1(Y',\mathcal{O}_{Y'})_s$ and of $H^1(Y', \mathcal{O}_{Y'})_n$ whose elements are $k$-linear combinations of the $\frac{y^{a+b}}{z^{a}x^{b}}$. Denote now $H^1(Y', \mathcal{O}_{Y'})$ by $V$ and recall that $V_s = F^m(V)$ for $m \gg 0$.

It is enough to show that the degree zero part of $c$ (considered as a polynomial in $t$), namely $\frac{y^{2pl}}{zx^{2pl-1}} - \frac{y^{2pl}}{z^{pl+1}x^{pl-1}}$, is not contained in $V_s$.
Indeed, if $c$ is contained in $V_s$ then we can write $c = \sum_{b\in B} \alpha_b b$, where $B$ is a basis of $V_s$ as above. As $c$ is a polynomial in $t$ the $\alpha_b$ are in $k[t]$. Hence, every homogeneous component (with respect to the grading induced by $t$) has to be contained in $V_s$.

Since $(p, 3pl -1) =1$ we may write $p^n = s(3pl-1) + 1$ for suitable $n, s \in \mathbb{N}$. Hence, a basis element $\frac{y^{a+b}}{z^ax^b}$ is mapped via $F^n$ to a sum of elements whose numerators are again $y^{a+b}$. Replacing $n$ by a sufficiently large multiple of $n$ we may assume that $V_s = F^n(V)$. Now one computes for $a+b = 2pl$ that \[F^n\left(\frac{y^{2pl}}{x^a z^b}\right) = \frac{y^{2pl} (x^{(3pl -1)p} + z^{(3pl-1)p} + x^{p^2l} z^{(2pl-1)p})^{2ls}}{x^{ap^n}z^{bp^n}},
\] and this is a sum of elements of the form $\frac{y^{2pl}}{x^{a'p}z^{b'p}}$ where $a' + b' = 2l$. In particular, the exponents of the denominators in the image are divisible by $p$ while this is not the case for the degree zero part of $c$. Hence, the degree zero part is not contained in $V_s$ and we conclude that it has some non-trivial nilpotent part.
\end{proof}

\begin{Theo}
\label{BundlesPgeq3}
Let $k$ be an algebraically closed field of characteristic $p \geq 2$, $1 \leq l \leq \lfloor\frac{p}{2}\rfloor$ an integer and $K$ an algebraic closure of $k(t)$.
Then there is a non-ordinary \'etale cover \[\varphi: X \longrightarrow  Y = \Proj k[t]_{(t-1)}[x,y,z,]/(x^{3pl-1} + y^{3pl-1} + z^{3pl-1} + x^{pl} z^{2lp-1})\] such that $\varphi^\ast \mathcal{S}(3) = \varphi^\ast(\Syz(x^2, y^2, z^2 + t x z)(3))$ on $X_t \times_{k(t)} K$ is stable and $F$-trivial and stable but not strongly semistable on $X_1$.
\end{Theo}
\begin{proof}
For $p =2$ we can take $X = Y$ since $\mathcal{S}_t$ is then actually $F$-trivial. So let $p \geq 3$.
Note that we have $H^1(Y_0, \mathcal{O}_{Y_0}) \otimes k(t) = H^1(Y_t, \mathcal{O}_{Y_t})$ since the family is trivial. The vector space $H^1(Y_0, \mathcal{O}_{Y_0})_s$ admits a basis $B$ whose elements are fixed by Frobenius. By Artin-Schreier theory (cf.\ e.\,g.\ \cite[III \S 4]{milne}) there is an \'etale cover $\varphi: X_0 \to Y_0$ such that $\varphi^\ast(b) = 0$ for all $b \in B$. We denote this morphism by $\varphi_0: X_0 \to Y_0$. Via base change (\cite[Proposition I.3.3 (c)]{milne}) we obtain an \'etale morphism $\varphi: X = X_0 \times \Spec k[t]_{(t-1)} \to Y$.

Denote by $c \in H^1(Y_t, \mathcal{O}_{Y_t})$ the cohomology class described in Lemma \ref{PSyzfuerallepgeq3}. We can write $c = c_s + c_n$, where $c_n \in H^1(Y_t, \mathcal{O}_{Y_t})_n$ and $c_s \in H^1(Y_t, \mathcal{O}_{Y_t})_s$. We then obtain $\varphi_t^\ast(c_s) =0$. In particular, $\varphi^\ast \mathcal{S}(3)$ is generically $F$-trivial.


Lemma \ref{PSyzfuerallepgeq3} yields that $\mathcal{S}(3)_t$ is stable. Hence, by \cite[proof of Proposition 6.8.\ (ii)]{balajiparameswarannarasimhanseshadrianalogue} $\varphi^\ast (\mathcal{S}(3)_t)$ is polystable, i.\,e.\ a direct sum of stable bundles. We want to prove that it is in fact stable.

First of all, note that the class $\varphi^\ast(c_n)$ is non-zero. For otherwise the $F$-trivial bundle on $Y$ defined by $c_n$ would be \'etale trivializable. This is impossible by \cite[Bemerkung 1.7]{langestuhler}. In particular, it follows that $F^\ast \varphi^\ast \mathcal{S}(3)_t$ is a non-split extension defined by $\varphi^\ast(c_n)$. This also shows that $X$ is non-ordinary. 

Assume now that $\varphi^\ast (\mathcal{S}(3)_t)$ is not stable, hence a direct sum of two line bundles $\mathcal{L}$ and $\mathcal{G}$. Since the determinant of $\mathcal{S}(3)_t$ is trivial we must have $\mathcal{G} = \mathcal{L}^\vee$. As its first Frobenius pull back admits a global section $\mathcal{L}$ has to be $p$-torsion. But at the same time $F^\ast \varphi^\ast \mathcal{S}(3)_t$ is the extension of $\mathcal{O}_X$ with itself defined by the non-trivial class $\varphi^\ast(c_n)$. This contradiction shows that $\varphi^\ast \mathcal{S}(3)_t$ has to be stable.

Since $\varphi$ is separable the pull back of $\mathcal{S}(3)_1$ to $X$ is still semistable. As $\varphi^\ast \mathcal{S}(3)_1$ is also not strongly semistable it is stable by Lemma \ref{SufficientStable}. 
\end{proof}

\begin{Bsp}
\label{Ftrivialnotopen}
Consider the case $p = 3$ and $l =1$. Then $p^2 = (3p-1) + 1$ so that $s = 1$. One can show that the degree zero part of $c$ is nilpotent in this case. However, the term of highest degree of $c$ with respect to $t$ is $\frac{y^{6}}{x^3 z^3}$ and $F^2(\frac{y^{6}}{x^3 z^3}) =  - \frac{y^6}{x^3z^3}$. Hence, $c$ is not $F$-nilpotent and the passage to an \'etale cover is necessary in this case.

It is also worth noting that in this case $\mathcal{S}(3)$ over $X \times \Spec k[t]$ is $F$-trivial over the special fiber $t =0$. Indeed, it is easily verified with a computer algebra system (e.\,g.\ \cite{CocoaSystem}) that $F^{2^\ast} \Syz(x^2, y^2, z^2)(3)$ over $X$ has two linearly independent global sections without zeros.
Namely,
\begin{align*}
s_1 &=(x^6z^3, y^6z^3, xy^8 - x^4z^5 + xz^8 + z^9),\\
s_2 &=(x^7y^2 - x^3y^2z^4 - y^2z^7, xy^8 + x^5z^4 + x^2z^7 - z^9,\\& x^6y^2z + x^5y^2z^2 + x^3y^2z^4 - x^2y^2z^5 + xy^2z^6 + y^2z^7).
\end{align*}
 Thus the $F$-triviality follows from Lemma \ref{LTrivial}.
\end{Bsp}

\begin{Ko}
Let $X_1$ and $k$ be as above. Denote by $K$ an algebraic closure of $k(t)$. Then the natural morphism \[h_{X_1}: \pi_1(X_1 \times_k \Spec K, x) \to \pi_1(X_1, x) \times_k \Spec K\]  of affine group schemes over $K$ is not an isomorphism.
\end{Ko}
\begin{proof}
Combine Theorem \ref{BundlesPgeq3}, Proposition \ref{NonConstantFamily} and \cite[Proposition 3.1]{methasubramanianfungroupscheme}.
\end{proof}

\section{$F$-trivial bundles in families}
In this last section we investigate in a more general context how $F$-trivial bundles behave in (not necessarily trivial) families.

The following basic proposition describes the behavior of $F$-triviality in families and is based on the usual argument for trivial bundles in a family.

\begin{Prop}
\label{Ftrivialfamilies}
Let $\pi: X \to Y$ be a projective morphism of noetherian schemes, where $Y$ is integral and defined over an algebraically closed field $k$ of positive characteristic $p > 0$. Assume that $\mathcal{S}$ is a coherent sheaf on $X$, flat over $Y$.

If $\mathcal{S}$ restricted to the generic fiber is trivialised by the $e$th Frobenius pull back then there is a non-empty open subset $U$ of $Y$ such that $\mathcal{S}_t$ is $F$-trivial for all $t$ in $U$.
Conversely, if there is $e \geq 0$ such that $F^{e^\ast}\mathcal{S}_u$ is trivial for all $u$ in a dense subset $U$ of $Y$ then $\mathcal{S}$ restricted to the generic fiber $\eta$ is trivialised by $F^{e}$.
\end{Prop}
\begin{proof}
Denote the rank of $\mathcal{S}$ by $n$ and denote the generic fiber by $\mathcal{S}_\eta$. Assume that $F^{e^\ast} \mathcal{S}_\eta \cong \mathcal{O}_{X_\eta}^n$. This isomorphism is given by $n$ global sections $s_1, \ldots, s_n$ without zeros which are linearly independent (cf. Lemma \ref{LTrivial}).

The issue is local on the base, so that we may assume that $Y = \Spec A$ is the spectrum of a noetherian integral domain.
We have that $H^0(X_y, F^{e^\ast}\mathcal{S}_y) = H^0(X, F^{e^\ast}\mathcal{S} \otimes k(y))$ for all $y \in Y$ (cf. \cite[Corollary 9.4]{hartshornealgebraic}). Moreover, since sheafification, tensor products and cohomology all commute with direct limits we obtain an inclusion $H^0(X, F^{e^\ast}\mathcal{S} \otimes A_f) \to H^0(X, F^{e^\ast}\mathcal{S} \otimes Q(A))$ and we may assume that the sections $s_1, \ldots, s_n$ are already defined over $A_f$. We therefore reduced to the situation that we have an injective morphism $\varphi: \mathcal{O}_X^n \to F^{e^\ast}\mathcal{S}$ which is an isomorphism over the generic fiber. The cokernel of $\varphi$ is supported on a (proper) closed subset of $X$. As $\pi$ is projective the image of the support is closed in $Y$. Moreover, it does not contain the generic point of $Y$. Hence, we obtain a non-empty open subset $U$ of $Y$ where we obtain an isomorphism on the fibers.

For the other direction, assume to the contrary that $\mathcal{S}$ is not trivialized by $F^e$ on the generic fiber, i.\,e.\ assume that $\dim_{K(Y)} H^0(X_\eta, F^{e^\ast}\mathcal{S}_\eta) \neq n$ or that any choice of $n$ global sections admits zeros. In the first case we obtain by semicontinuity that $\dim_{K(Y)} H^0(X_\eta, F^{e^\ast}\mathcal{S}_\eta) \leq n -1$. But the set where the dimension of global sections is $\leq n-1$ is open (and non-empty) by semicontinuity. Hence, the set of points where the dimension is strictly bigger is a proper closed subset. But then it cannot contain a dense subset. We conclude that $\dim_{K(Y)} H^0(X_\eta, F^{e^\ast}\mathcal{S}_\eta) = n$. As in the other direction we therefore obtain an injective morphism $\mathcal{O}_X^n \to F^{e^\ast}\mathcal{S}$. As before the cokernel is supported on a closed subset. Since, by assumption, the image of this set is a proper subset its image cannot contain the generic point of $Y$.
\end{proof}

Based on this result it seems natural to ask whether it is sufficient to have a (not necessarily bounded) family $e_t$ such that $F^{e_t^\ast}\mathcal{S}_t$ is trivial for all $t$ in a dense subset to draw the same conclusion. The answer is no, as the following example shows:

\begin{Bsp}
Let $k$ be an algebraically closed field of characteristic $p > 0$ and let $X$ be an ordinary elliptic curve. Fix a Weierstra\ss\ equation $f(x,y,z)$ for $X$ so that $O = (0,1,0)$ is the only point in $V_+(z)$. Let now $Y = V(f(s,t, 1)) \subseteq \Spec k[s,t]$ and consider the trivial family $X \times_k Y \to Y$. Then $(s,t,1) - O$ is a (relative) Weil divisor of degree zero -- denote the associated line bundle by $\mathcal{L}$.

Since $X$ is ordinary there are points $P_n = (a_n, b_n, 1)$ in $Y$ such that the line bundle $\mathcal{T}_n$ associated to $P_n - O$ is of order $p^n$ in $\Pic^{0}(X)$. Specialising $s \mapsto a_n, t \mapsto b_n$ we obtain a sequence of line bundles trivialised by the $e$th Frobenius, where $e$ cannot be bounded. In particular, it follows from Proposition \ref{Ftrivialfamilies} that the line bundle on the generic fiber of $X \times Y$ cannot be $F$-trivial.
\end{Bsp}

Note that Example \ref{Ftrivialnotopen} shows that it is not enough to require that $F^{e^\ast} \mathcal{S}_u$ be trivial for all $u$ in a non-empty subset minus the generic point. In particular, $F$-triviality is not an open condition. A simple instance of this fact is the following well-known

\begin{Bsp}
Let $Y$ be a geometrically integral smooth projective curve defined over a perfect field $k$ of positive characteristic and assume that Frobenius does not act nilpotently on $H^1(Y, \mathcal{O}_X)$. Fix a cohomology class $c$ in $H^1(Y, \mathcal{O}_Y)_s$. Denote by $Y_t$ the base change of $Y$ to $k(t)$. The cohomology class $t \cdot c$ yields a rank $2$ vector bundle $\mathcal{S}$ over $Y \times \mathbb{A}^1_k$. Over the generic fiber $\mathcal{S}$ is not $F$-trivial. However, over the special fiber $t= 0$ the bundle is trivial.

A somewhat more sophisticated example is obtained by considering $t \cdot c + n$, where $0 \neq n \in H^1(Y, \mathcal{O}_Y)_n$ -- this exists provided that, in addition, $Y$ is not ordinary. Then one obtains a bundle which is generically not $F$-trivial but trivialised by a power of Frobenius ($\geq 1$) over the fiber $t= 0$.
\end{Bsp}

In closing we provide an example of a non-trivial family that exhibits a similar behavior to the examples provided in sections 2 and 3.

\begin{Bsp}
Let $k$ be a field of characteristic $3$ and consider the family $T = V_+(x^5 + y^5 + z^5 + tx^2y^3) \subset \mathbb{P}^2_{k[t]}$ whose members we denote by $X_t$. Fix the syzygy bundle $\mathcal{S} = \Syz(x^2,y^2,z^2)$ on $\mathbb{P}^2_{k[t]}$. We claim that  $\mathcal{S}(3)$ restricted to the generic fiber of $T$ is $F$-trivial while $\mathcal{S}(3)$ restricted to the special fiber $X_0$ is not $F$-trivial but stable.

One verifies with the help of a computer algebra system that $F^{3^\ast}\mathcal{S}_t$ admits two linearly independent syzygies without zeros of total degree $81$. So that $F^{3^\ast}(\mathcal{S}(3)_t)$ is trivial by Lemma \ref{LTrivial}

Moreover, on the special fiber $t =0$ (i.\,e.\ on the Fermat quintic) $F^\ast \mathcal{S}$ admits the syzygy $s = (zy, xz, xy)$ of total degree $8$. Note that $s$ has no zeros. It follows from \cite[Lemma 2.3]{brennerstaeblergeometricdeformations} and Lemma \ref{SufficientStable} that $\mathcal{S}(3)$ is stable but not strongly semistable.
\end{Bsp}

\begin{Bem}
In \cite[Remark 4.8]{brennerstaeblergeometricdeformations} Brenner and the present author have shown that for a field of characteristic $2$ the syzygy bundle $\Syz(x,y,z)$ on $T \subset \mathbb{P}_{k[t]}$ admits a similar behavior.
\end{Bem}

\section*{Acknowledgements}
This paper arose from discussions with Holger Brenner in relation to our joint paper \cite{brennerstaeblergeometricdeformations}. In particular, I thank him for sparking my interest in this problem and for several useful discussions. Furthermore, I thank Manuel Blickle for useful discussions and the referee for a careful reading of an earlier draft and useful comments.

\bibliography{bibliothek.bib}
\bibliographystyle{amsplain}
\end{document}